\documentclass{baustms}

\citesort

\theoremstyle{cupthm}
\newtheorem{Theorem}{Theorem}[section]

\newtheorem{Lemma}[Theorem]{Lemma}
\theoremstyle{cupdefn}
\newtheorem{Definition}[Theorem]{Definition}
\theoremstyle{cuprem}
\newtheorem{Remark}[Theorem]{Remark}
\numberwithin{equation}{section}

\newtheorem{Conjecture}[Theorem]{Conjecture}
\newtheorem{Example}[Theorem]{Example}

\begin{document}
\def\F{{\mathbb F}}

\def\A{{\cal A}}

\def\L{{\cal L}}

\def\SS{{\cal S}}

\def\B{{\cal B}}

\def\K{{\mathbb K}}

\def\C{{\cal C}}

\def\D{{\cal D}}

\def\R{{\cal R}}

\def\P{{\cal P}}
\def\Z{{\mathbb Z}}
\def\T{{\cal T}}
\def\X{{\cal X}}
\def\N{{\cal N}}
\def\FF{{\cal F}}
\def\DD{{\mathbb D}}
\def\RR{{\mathbb R}}
\def\NN{{\mathbb N}} \def\CC{{\mathbb C}} \def\ZZ{{\mathbb Z}}
\def\chr{{\rm char}\,} \def\Re{{\rm Re}\,} \def\Im{{\rm Im}\,}
\newcommand{\diag}{{\text {diag}}}

\runningtitle{An upper bound on the length of an algebra}
\title{An Upper Bound on the Length of an Algebra and Its Application to the Group Algebra of the Dihedral Group}
\author[1]{M. A. Khrystik}
\address[1]{HSE University, Faculty of Computer Science, Moscow, 101000, Russia.}
\address[2]{Moscow Center of Fundamental and Applied Mathematics, Moscow, 119991, Russia.\email{good\_michael@mail.ru}}

\authorheadline{M. A. Khrystik}


\support{This research was supported by Russian Science Foundation, grant 20-11-20203, https://rscf.ru/en/project/20-11-20203/}

\begin{abstract}
Let $\A$ be an $\F$-algebra and let $\SS$ be its generating set. The length of $\SS$ is the smallest number $k$ such that $\A$ equals the $\F$-linear span of all products of length at most $k$ of elements from $\SS$. The length of $\A$, denoted by $l(\A)$, is defined to be the maximal length of its generating set. In this paper, it is shown that the $l(\A)$ does not exceed the maximum of $\dim \A / 2$ and $m(\A)-1$, where $m(\A)$ is the largest degree of the minimal polynomial among all elements of the algebra $\A$. For arbitrary odd $n$, it is proven that the length of the group algebra of the dihedral group of order $2n$ equals $n$.
\end{abstract}


\maketitle

\section{Introduction}

All algebras considered in this paper are {\bf associative finite-dimensional algebras with an identity over a field}. First, we recall the notion of the {\em length} of the algebra $\A$.

Let $\A$ be an algebra. Any product of a finite number of elements from a finite subset $\SS \subset \A$ is called a word over the alphabet $\SS$. The length of a word equals the number of letters in this product that are different from $1_{\A}$. We consider $1_{\A}$ to be an empty word of length 0.

If $\SS$ is a generating system (or a generating set) of the algebra $\A$, i.e., $\A$ is the minimal subalgebra of $\A$ containing $\SS$, then any element of the algebra $\A$ can be expressed as a linear combination of words over $\SS$. The minimal $k$ such that all elements of $\A$ can be expressed using words of length no more than $k$ is called the length of the generating system $\SS$. The length of the algebra $\A$ is defined as the maximum length among its generating systems and will be denoted by $l(\A)$ (see definition \ref{alg_len}).
In defining the length of algebra $ \A $, we consider the set of {\bf all} generating systems for $ \A $. This explains the difficulty of calculating the length even for classical algebras.

The general problem of calculating the length was first formulated by A.~Paz in 1984 for the full matrix algebra $M_n(\F)$ over a field in \cite{Paz} and still remains open.

\begin{Conjecture}[\cite{Paz}]
Let $\F$ be an arbitrary field. Then $l(M_n(\F))=2n-2.$
\end{Conjecture}

A nontrivial upper bound on $l(\A)$ in terms of $\dim \A$ and $m(\A)$ (the largest degree of the minimal polynomial among all elements of the algebra $\A$) was obtained in \cite{Pap} by C.~Pappacena. The study of upper bounds on length in these terms will be continued in this paper.

Calculating the length in general is a rather difficult task.
The main algebraic properties of the length function were studied by O.V.~Markova in the work \cite{OVM}.

The question of calculating the lengths of group algebras is of particular interest. Due to their matrix representations, solving this question is closely linked to solving Paz's problem.
For group algebras of small-order groups it is possible to calculate the length precisely over arbitrary fields. For the permutation group $S_3$, Klein four-group $K_4$, and quaternion group $Q_8$, the lengths were found by A.E. Guterman and O.V. Markova in \cite{GutM18,GutM19}.

Systematic study of the general problem of finding the lengths of group algebras of finite abelian groups was dedicated to the joint works of the author with A.E. Guterman and O.V. Markova \cite{GMK1,GutKhM20p2}. The works of O.V.~Markova \cite{Mar20} and the author \cite{Kh23} continued the study of the lengths of group algebras of finite abelian groups in the modular case.

Studying all non-abelian groups appears to be too difficult due to the diversity of their structure. Therefore, it is proposed to study the length function separately for families of classic non-abelian groups. Thus, in the joint work of the author with O.V. Markova \cite{KhMar20}, the study of the lengths of group algebras of dihedral groups began, and the length was calculated in the semisimple case. This series of groups in the semisimple case is a natural next step after the abelian case. Indeed, for group algebras of abelian groups in the decomposition into a direct sum of matrix algebras all terms are one-dimensional, whereas the sizes of the matrix algebras in the decomposition into a direct sum of group algebras of dihedral groups do not exceed two. The work \cite{KhMar20POMI} continued the study of the lengths of group algebras of dihedral groups of order $2^k$ and calculated their length in the modular case. This paper will consider the length of the group algebra of the dihedral group over an arbitrary field.

In Section \ref{main_def}, the main definitions and notations of the considered theory are introduced.

In Section \ref{genbound}, the upper bound on the length is proven.

In Section \ref{lendih}, the concept of bicirculant algebra is introduced and studied, in particular, its length is calculated. A bicirculant representation of the group algebra of the dihedral group is constructed and its properties are studied. Using the bicirculant representation, $l(\F \mathcal D_n)$ and $m(\F \mathcal D_n)$ are estimated.

\section{Main Definitions and Notations}\label{main_def}

Denote by $\langle S \rangle$
the linear span (the set of all finite
linear combinations with coefficients from $\F$) of a subset $S$ of some vector space over $\F$.

Let $B=\{b_1,\ldots,b_m\}$ be a non-empty finite set (alphabet). Finite sequences
of letters from $B$ are called words.  Let $B^*$ denote the set
of all words in the alphabet $B$, $F_B$ be the free semigroup over
the alphabet $B$, i.e. $B^*$
 with the operation of concatenation.

\begin{Definition}\label{word_len}  {\em The length\/} of the word
$b_{i_1}\ldots b_{i_t}$, where $b_{i_j}\in B$, is equal to $t$. We
will consider $1$ (the empty word) a word from the elements $B$ {\em of
length
$0$\/}.
\end{Definition}

Let $B^i$ denote the set
of all words in the alphabet $B$ of length no greater than $i$, $i\geq 0$. Then
by
$B^{=i}$ denote the set of all words in the alphabet $B$ of length
equal to $i$, $i\geq 1$.

\begin{Remark} Products of elements from the generating set $\SS$ can
be considered as images of elements of the free semigroup $F_{\SS}$
under the natural homomorphism, and they can also be called words
from the generators and use the natural notations $\SS^i$ and
$\SS^{=i}$.

\end{Remark}

Denote by $\L_i(\SS)$
the linear span of words from $\SS^i$. Note that $\L_0(\SS)=\langle
1_{\A}\rangle=\F$.
Let also $\L(\SS)=\bigcup\limits_{i=0}^\infty \L_i(\SS)$
denotes the linear span of all words in the alphabet $\SS=\{a_1,\ldots,
a_k\}$.

\begin{Definition}\label{sys_len} {\em The length of a generating
system
$\SS$\/} of algebra $\A$ is $l(\SS)=\min\{k\in \ZZ_+: \L_k(\SS)=\A\}$. \end{Definition}

\begin{Definition}\label{alg_len}
 {\em The length of an algebra $\A$} is $l(\A)=\max \{l(\SS):  \L(\SS)=\A\}$. \end{Definition}

Let $\A$ be an algebra, $\tau \in \A$. Denote the minimal polynomial of $\tau$ by $\mu_{\tau}(x)$. Then $m(\tau)=\deg \mu_{\tau}(x)$, $m(\A)=\max_{\tau \in \A} m(\tau)$.

Denote by $\F G$ or $\F[G]$ the group algebra of the group $G$ over the field $\F$, $E_{i,j}$ for the matrix unit, $\mathcal D_n$ for the dihedral group of order $2n$, $S_n$ for the symmetric group.

\begin{Definition}\label{equiv}
We say that two words $u$ and $v$ of length $i$ from the generators are {\em equivalent}, if $u-\alpha v\in \L_{i-1}(\SS)$ for some nonzero $\alpha \in \F$. We will use the notation $u\sim v$ in this case.
\end{Definition}

\begin{Definition}
We say that a word $u$ of length $i$ from the generators {\em reducible} if $u\in \L_{i-1}(\SS)$. Otherwise, we will call the word {\em irreducible}.
\end{Definition}

\section{General Bound on Length}\label{genbound}

\subsection{Equivalence of Words}\

Before proceeding to prove the main statement of the section let us note some properties of the introduced concept of word equivalence as it is significantly used in the proof of this statement.

\begin{Lemma}\label{eqrel}
Equivalence of words is an equivalence relation on the set of words.
\end{Lemma}

\begin{proof}

{\em Reflexivity.} $u-\alpha u \in \L_{i-1}(\SS)$ with $\alpha=1.$

{\em Symmetry.} Let  $u-\alpha v \in \L_{i-1}(\SS)$. Then, by multiplying the element  $u-\alpha v$ by $-\alpha^{-1}$, we get  $v-\alpha^{-1} u \in \L_{i-1}(\SS).$

{\em Transitivity.} Let  $u-\alpha_1 v \in \L_{i-1}(\SS)$, $v-\alpha_2 w \in \L_{i-1}(\SS)$. Then, by adding the second element multiplied by $\alpha_1$ to the first one, we obtain $u-\alpha_1 \alpha_2 w \in \L_{i-1}(\SS).$

\end{proof}

\begin{Lemma}\label{eqred}
Let $u \sim v$. Then $u$ is reducible if and only if $v$ is reducible.
\end{Lemma}

\begin{proof}
Straightforward.

\end{proof}

\begin{Lemma}\label{eqsub}
Let the word $u$ be irreducible. Then any subword of $u$ is irreducible.
\end{Lemma}

\begin{proof}
Straightforward.

\end{proof}

\begin{Lemma}\label{eqrep}
Let the word $w$ of length $i$ contain a subword $u$ of length $j$, $u \sim v$. Then $w \sim w'$, where $w'$ is a word obtained from $w$ by replacing the subword $u$ with the subword $v$.
\end{Lemma}

\begin{proof}
By condition, $u-\alpha v \in \L_{j-1}(\SS)$, $w=w_1uw_2$, for some words $w_1$, $w_2$. Then, by multiplying the expression $u-\alpha v$ on the left by $w_1$ and on the right by $w_2$, we get $w-\alpha w' \in \L_{i-1}(\SS).$
\end{proof}

\subsection{Estimating $l(\A)$ Using $\dim \A$ and $m(\A)$}\

\begin{Theorem}\label{ldm}
 Let $\A$ be an associative finite-dimensional algebra with an identity. Then
$$l(\A)\leq max\{m(\A)-1,\frac{\dim\A}{2}\}.$$
\end{Theorem}

\begin{proof} Let $l(\A)\geq m(\A)$ (otherwise the statement is proven). Let $\SS$ be a generating set of length $l(\A)$ of the algebra $\A$ (in the case of other generating sets the length of the algebra will be no greater). Consider an irreducible word $a_1a_2\cdots a_{l(\A)}$ of length $l(\A)$ in the alphabet $\SS$ (such exists by definition of the length of the algebra). We will prove that $\forall k\in [1,l(\A)-1]$ it holds that $\dim\L_k(\SS)-\dim \L_{k-1}(\SS)\geq 2.$

We will reason by contradiction. Suppose  $\exists k\in [1,l(\A)-1]$ such that  $\dim\L_k(\SS)-\dim \L_{k-1}(\SS)=1$ (this difference cannot be zero by definition of the length of the algebra). We will break the reasoning into steps and lead it to a contradiction.

{\em First step.} The word  $a_1a_2\cdots a_{l(\A)}$ is irreducible. Therefore, its subword $a_1a_2\cdots a_k$ is irreducible by Lemma \ref{eqsub}. By assumption $a_2a_3\cdots a_{k+1} \sim a_1a_2\cdots a_k$ (here we use the fact that $k$ is no greater than $l(\A)-1$). Indeed, if this were not the case, we would get $\dim\L_k(\SS)-\dim \L_{k-1}(\SS)\geq 2$, since the dimension would increase by at least 2 due to these two words. Thus,  $a_1a_2\cdots a_{l(\A)} \sim  a_2
a_3\cdots a_k a_{k+1} a_{k+1} a_{k+2} \cdots a_{l(\A)}$ by Lemma \ref{eqrep}. Therefore, the word $ a_2 a_3\cdots a_k a_{k+1} a_{k+1} a_{k+2} \cdots a_{l(\A)}$ is irreducible.

{\em Second step.} Now consider the irreducible word $ a_2 a_3\cdots a_k a_{k+1} a_{k+1} a_{k+2} \cdots a_{l(\A)}$ of length $l(\A)$ obtained in the previous step. By reasoning similarly (considering subwords of length $k$ starting from the first and second letters), we will get rid of the letter $a_2$ similarly to how we got rid of the letter $a_1$ in the first step. We obtain that the word $ a_3 a_4\cdots a_k a_{k+1} a_{k+1} a_{k+1} a_{k+2} \cdots a_{l(\A)}$ is irreducible.

After conducting $k$ steps of this reasoning, we obtain that the word  $a_{k+1}\cdots a_{k+1} a_{k+2} \cdots a_{l(\A)}$ of length $l(\A)$ is irreducible. Now we can proceed to the last step and obtain a contradiction.

{\em $(k+1)$-st step.} The word  $a_{k+1}^{k+1} a_{k+2} \cdots a_{l(\A)}$ is irreducible.  Therefore, its subword $a_{k+1}^{k}$ is irreducible. By assumption, all words of length $k$ are expressed through the word $a_{k+1}^{k}$ and words of shorter length. Thus, $a_1a_2\cdots a_{l(\A)} \sim a_{k+1}^{l(\A)}$. Therefore, the word $a_{k+1}^{l(\A)}$ is irreducible and  $l(\A)< m(\A)$. Contradiction.

We return to the proof of the main statement. Represent the dimension of the algebra in the following form $\dim \A=\dim\L_{l(\A)}(\SS)=(\dim\L_{l(\A)}(\SS)-\dim\L_{l(\A)-1}(\SS))+(\dim\L_{l(\A)-1}(\SS)-\dim\L_{l(\A)-2}(\SS))+\cdots+(\dim\L_1(\SS)-\dim\L_0(\SS))+\dim\L_0(\SS)$. The first term of this sum is not less than 1, the last one equals 1, and all the others are not less than 2. Thus,  $\dim \A \geq 1+2(l(\A)-1)+1$. Therefore, $l(\A) \leq \frac{\dim\A}{2}$. Thus, $l(\A)\leq max\{m(\A)-1,\frac{\dim\A}{2}\}.$

\end{proof}

\subsection{Comparison with Other Estimates}\

In conclusion of this section we will compare the obtained bound with other similar bounds.

Let us compare the obtained bound with the following bound presented in the joint work of the author with O.V. Markova.

\begin{Lemma}[{\cite[Lemma 2.10]{KhMar20POMI}}]\label{d<m+4}
Let $\mathcal A$ be an $\F$-algebra, $\dim\mathcal A\leq m(\mathcal A)+4$, $m({\mathcal A}) \geq 3$. Then $l(\mathcal A) \leq m(\mathcal A)$.
\end{Lemma}

Since $m(\A)-1$ is unequivocally less than $m(\A)$, we see that the new estimate will be worse than the estimate from Lemma \ref{d<m+4} only if $\dfrac{\dim\A}{2} \geq m(\A)+1$ (that is, if $\dim\A \geq 2m(\A)+2$). Also, by the condition of Lemma \ref{d<m+4} it must be fulfilled that $\dim\mathcal A\leq m(\mathcal A)+4$. From the last two inequalities, it follows that $m(\A) \leq 2$. But in the condition of Lemma \ref{d<m+4} it is also required that $m({\mathcal A}) \geq 3$. Therefore, the new bound is better in any case.

Next we will compare with the following Pappacena's estimate.

\begin{Theorem}[{\cite[Theorem 3.1]{Pap}}]\label{Pap}
Let $\A$ be any algebra. Then
$ l(\A)< f(\dim \A,m(\A))$, where
$$f(d,m)=m\sqrt{\frac{2d}{m-1}+\frac{1}{4}}+\frac{m}{2}-2.$$
\end{Theorem}

Since $\dim\A \geq m(\A)-1$, we have $m\sqrt{\dfrac{2d}{m-1}+\dfrac{1}{4}}+\dfrac{m}{2}-2 \geq m\sqrt{\dfrac{9}{4}}+\dfrac{m}{2}-2 = 2m-2.$ Since $m(\A)-1$ is less than $2m(\A)-2$, we see that the new estimate will be worse than Pappacena's estimate only if $\dfrac{\dim\A}{2} > 2m(\A)-2$ (that is, if $\dim\A > 4(m(\A)-1)$). That is, the new bound can be worse than Pappacena's bound only if the dimension of the algebra is 4 times greater than the expression $m(\A)-1$. In particular, the new estimate is unequivocally better when considering group algebras of dihedral groups, which will be discussed in the next section. However, Theorem \ref{ldm} may give a more accurate estimate than Theorem \ref{Pap} even if $\dim\A \leq 4(m(\A)-1)$. Let us show that by the following example.

\begin{Example}
Let $\A = M_3(\mathbb F)$. Then $\dim \A = 9$, $m(\A)=3$. Theorem \ref{Pap} gives an estimate $l(\A) \leq 8$. Theorem \ref{ldm} gives an estimate $l(\A) \leq 4$, which corresponds to the value $l(M_3(\mathbb F))$ in Paz's conjecture.
\end{Example}

\section{Calculating $l(\F\D_n$)}\label{lendih}

\subsection{Bicirculant Algebra}\

Let us consider two matrices. The circulant $A_n=E_{n,1}+E_{1,2}+\cdots+E_{n-1,n}$ and the anti-circulant $B_n=E_{1,n}+\cdots +E_{n,1}$. 
$$
A_n=
\begin{pmatrix}
0 & 1 & 0 &\ldots & 0\\
0 & 0 & 1 &\ldots & 0\\
0 & 0 & 0 &\ldots & 0\\
\vdots& \vdots & \vdots &\ddots & \vdots\\
0 & 0 & 0 &\ldots & 1\\
1 & 0 & 0 &\ldots & 0
\end{pmatrix}
,\quad
B_n=
\begin{pmatrix}
0 & 0 &\ldots & 0 & 1\\
0 & 0 &\ldots & 1 & 0\\
\vdots& \vdots & \ddots &\vdots & \vdots\\
0 & 0 &\ldots & 0 & 0\\
0 & 1 &\ldots & 0 & 0\\
1 & 0 &\ldots & 0 & 0
\end{pmatrix}.
$$

Let us define the algebra generated by these two matrices.

\begin{Definition} {\em The algebra of bicirculants of order n} over the field $\F$ is $\C_n(\F)=\L(\{A_n,B_n\})$.
\end{Definition}

Let us study the structure of this algebra.

\begin{Lemma}\label{bcrel}
 $A_n^n=E$, $B_n^2=E$, $A_nB_n=B_nA_n^{n-1}$.
\end{Lemma}

\begin{proof} The equalities are checked directly by multiplying matrices.
\end{proof}

\begin{Lemma}\label{bcdim}
$\dim \C_n(\F)=\begin{cases} 2n-2,\ \mbox{for even}\; n;\\
2n-1, \ \mbox{for odd}\; n.
\end{cases}$
\end{Lemma}

\begin{proof} Due to Lemma \ref{bcrel} we may consider that $\C_n(\F)=\C_n'(\F)+\C_n''(\F)$, where $\C_n'(\F)=\langle E,A_n,A_n^2,\dots,A_n^{n-1}\rangle$, $\C_n''(\F)=\langle B_n,B_nA_n,B_nA_n^2,\dots,B_nA_n^{n-1}\rangle$. Note that $\C_n'(\F)$ is nothing else but the space of circulants, and $\C_n''(\F)$ is the space of anti-circulants, each of which has a dimension of $n$.

The basis of the intersection of the spaces $\C_n'(\F)$ and $\C_n''(\F)$ in the odd case is the matrix in which each element equals 1, and in the even case, the basis will be the following two matrices

$$
\begin{pmatrix}
1 & 0 & 1 &\ldots & 0\\
0 & 1 & 0 &\ldots & 1\\
1 & 0 & 1 &\ldots & 0\\
\vdots& \vdots & \vdots &\ddots & \vdots\\
1 & 0 & 1 &\ldots & 0\\
0 & 1 & 0 &\ldots & 1
\end{pmatrix}
 \ \mbox{and }
\begin{pmatrix}
0 & 1 & 0 &\ldots & 1\\
1 & 0 & 1 &\ldots & 0\\
0 & 1 & 0 &\ldots & 1\\
\vdots& \vdots & \vdots &\ddots & \vdots\\
0 & 1 & 0 &\ldots & 1\\
1 & 0 & 1 &\ldots & 0
\end{pmatrix}.
$$

Thus, the statement of the lemma follows from the formula for the dimension of the sum of subspaces.
\end{proof}

\begin{Theorem}\label{bclen}
$l(\C_n(\F))=n-1.$
\end{Theorem}

\begin{proof} Let us first prove the lower bound $l(\C_n(\F))\geq n-1.$ Consider a generating set $\SS=\{u,v\}$, where $u=B_n, v=A_nB_n$. This is indeed a generating set, as $\C_n(\F)=\L(\{A_n,B_n\})=\L(\{vu,u\})\subseteq \L(\{u,v\})=\L(\{B_n,A_nB_n\})\subseteq \L(\{A_n,B_n\})=\C_n(\F)$. At the same time, $u^2=v^2=E$, meaning that there are no more than two irreducible words of each length (of the form $uvuv\dots$ and $vuvu\dots$). Thus, $\dim\L_{n-2}(\SS)=(\dim\L_{n-2}(\SS)-\dim\L_{n-3}(\SS))+(\dim\L_{n-3}(\SS)-\dim\L_{n-4}(\SS))+\cdots+(\dim\L_1(\SS)-\dim\L_0(\SS))+\dim\L_0(\SS)\leq 2(n-2)+1<\dim\C_n(\F)$, from which it follows that the length of the algebra is at least $n-1$.

The upper bound $l(\C_n(\F))\leq n-1$ follows from Theorem \ref{ldm}. Indeed, by the Cayley-Hamilton theorem, $m(\C_n(\F))\leq n$. By Lemma \ref{bcdim}, $\dim \C_n(\F)\leq 2n-1$.  Applying Theorem \ref{ldm}, we obtain the inequality $l(\C_n(\F)) \leq max\{n-1,\frac{2n-1}{2}\}$. This completes the proof.

\end{proof}

\subsection{Bicirculant Representation of $\F\D_n$}\label{bcsect}\

Let us number the vertices of a regular $n$-gon. Let $d\in \D_n$ map the vertex $i$ to the vertex $\sigma(i)$ $\forall i$, where $\sigma\in S_n$. Then we can consider a group homomorphism, defining its values on elements of $\D_n$ by the rule $f(d)=\sigma$, and then extend it to an algebra homomorphism $f:\F\D_n\rightarrow \F S_n$ by linearity.

Let us now consider a group homomorphism $g: S_n \rightarrow M_n(\{0,1\})$, which maps a permutation from $S_n$ into the corresponding permutation matrix. We extend it to an algebra homomorphism $g: \F S_n \rightarrow M_n(\F)$ by linearity.

Note that the composition $g\circ f$ defines a linear representation of the algebra $\F\D_n$. This representation is called the {\em bicirculant representation} in this paper. Let us study some properties of this composition.

\begin{Lemma}\label{imgf}
$\Im g\circ f = \C_n(\F)$.
\end{Lemma}

\begin{proof}
Let $a$ be the rotation by an angle $\frac{2\pi}{n}$, $b$ be the symmetry about the axis passing through the vertex $\left[\dfrac{n}{2}\right]+1$. Then $\F \D_n=\langle e, a, a^2, \dots, a^{n-1}, b, ba, \dots , ba^{n-1} \rangle$.

It is easy to notice that $g\circ f(a)=A_n$, $g\circ f(b)=B_n$. Since $g\circ f$ is a homomorphism, $g\circ f(b^ia^j)=B_n^iA_n^j$, from which the statement of the lemma follows.
\end{proof}

\begin{Lemma}\label{kergf}
$\ker g\circ f = \langle e+a+\cdots+a^{n-1}-b-ba-\cdots-ba^{n-1} \rangle$, for odd n. $\ker g\circ f = \langle e+a^2+\cdots+a^{n-2}-b-ba^2-\cdots-ba^{n-2}, a+a^3+\cdots+a^{n-1}-ba-ba^3-\cdots-ba^{n-1}\rangle$, for even n.
\end{Lemma}

\begin{proof} The dimension of the kernel is established using Lemmas \ref{bcdim} and \ref{imgf}. The fact that the specified elements lie in the kernel and are linearly independent (in the case of even $n$) is checked directly.
\end{proof}

\subsection{Length of $\F\D_n$}\

First, let us present known results about the length of $\F\D_n$.

\begin{Lemma}[{\cite[Lemma 2.1]{KhMar20}}]\label{Died}
 Let ${\cal D}_{n}$ be the dihedral group of order $2n$, $n\geq3$, $\F$ be an arbitrary field. Then $l(\F{\cal D}_{n})\geq n$.

\end{Lemma}

\begin{Theorem}[{\cite[Theorem 1.15]{KhMar20}}]\label{Died_n} Let $\F$ be a field such that $\chr \F$ does not divide $2n$. Then $l(\F {\cal D}_n)=n$, for $n\geq 3$.
\end{Theorem}

\begin{Theorem}[{\cite[Theorem 4.10]{KhMar20POMI}}]
Let $\chr \F = 2$, $k\geq 2$. Then $l(\F\D_{2^k})=2^k$.\label{lfd2k}
\end{Theorem}

In this paper, we will try to generalize the last two theorems, namely, to eliminate the condition on the field.

Hereinafter in the work, it is assumed that $n\geq 3$.

Let us prove the main result of the section.

In the proof of the following lemma the author uses the idea of proving Lemma 3.11 from \cite{GutM18}.

\begin{Lemma}\label{sur}
Let there exist a surjective homomorphism of algebras $\varphi:\cal A \rightarrow \cal B$. Then $$l(\mathcal A) \leq l (\mathcal B)+\dim \mathcal A -\dim \cal B.$$
\end{Lemma}

\begin{proof}
Consider an arbitrary generating set $\SS=\{a_1,\dots,a_k\}$ of the algebra $\A$.

Since the homomorphism $\varphi$ is surjective, we see that the set $\SS_{\mathcal{B}}=\{c_1=\varphi(a_1),\dots,c_k=\varphi(a_k)\}$ is a generating set of the algebra $\mathcal{B}$. Therefore, $\dim\L_{l(\mathcal{B})}(\SS_{\B})=\dim \B$. On the other hand, $\L_{l(\B)}(\SS_{\B})=\L_{l(\B)}(\varphi(\SS))=\varphi(\L_{l(\B)}(\SS))$. Therefore, $\dim\L_{l(\B)}(\SS)\geq \dim \varphi (\L_{l(\B)}(\SS))=\dim \L_{l(\B)}(\SS_{\B})=\dim \B$.

Since the dimensions $\L_{i}(\SS)$ must increase with $i$ until stabilization, we have $\dim\L_{l(\B)+\dim \A - \dim \B}(\SS)\geq \dim \B + (\dim \A - \dim \B) = \dim \A$. At the same time, the minimal $i$ such that $\dim \L_{i}(\SS) = \dim \A$, by definition, is $l(\SS)$. Due to the arbitrariness of $\SS$, we obtain $l(\mathcal A) \leq l (\mathcal B)+\dim \mathcal A -\dim \cal B$.

\end{proof}

\begin{Theorem}\label{lendn}
Let ${\cal D}_{n}$ be the dihedral group of order $2n$, $n\geq3$, $\F$ be an arbitrary field. Then\\
$l(\F{\cal D}_{n})=\begin{cases} n,\ \mbox{for odd}\; n;\\
n \ \mbox{or}\; n+1, \ \mbox{for even}\; n.
\end{cases}$
\end{Theorem}

\begin{proof}

The lower bound is given by Lemma \ref{Died}. Let us prove the upper bound.

From Theorem \ref{bclen} it follows that $l(\C_n(\F))=n-1$. From Lemma \ref{bcdim} it follows that $\dim \C_n(\F)=2n-1$ for odd $n$, $\dim \C_n(\F)=2n-2$ for even $n$. Consider the homomorphism of algebras $g \circ f : \F \mathcal{D}_n \rightarrow \C_n(\F)$, described in Section \ref{bcsect}. Since by Lemma \ref{imgf} the homomorphism $g \circ f$ is surjective, we can apply Lemma \ref{sur} and get the upper bound $l(\F \mathcal{D}_n) \leq l (\C_n(\F))+\dim \F \mathcal{D}_n -\dim \C_n(\F)$. Then application of Theorem \ref{bclen}, Lemma \ref{bcdim} and the fact that $\dim \F \mathcal{D}_n = 2n$ completes the proof.

\end{proof}

\begin{Remark}
Despite the fact that among the possible values of $l(\F{\cal D}_{n})$ there is $n+1$, no real examples of algebras with this length have been found (and are not expected given Theorem \ref{Died_n}). The developed technique allows finding the exact value only for odd $n$, however, the obtained result is a noticeable advancement in the study of the lengths of group algebras of dihedral groups, demonstrating the usefulness of the bound proven in Theorem \ref{ldm} and the bicirculant representation.
\end{Remark}

\subsection{Bound for $m(\F{\cal D}_{n})$}\

Using the bicirculant representation, we get an estimate of $m(\F{\cal D}_{n})$.

\begin{Theorem}\label{mdn}
Let ${\cal D}_{n}$ be the dihedral group of order $2n$, $n\geq3$, $\F$ be an arbitrary field. Then\\
$m(\F{\cal D}_{n})\leq \begin{cases} n+1,\ \mbox{for odd}\; n;\\
 n+2, \ \mbox{for even}\; n.
\end{cases}$
\end{Theorem}

\begin{proof}
Let $\tau \in \F{\cal D}_{n}$, $g \circ f : \F \mathcal{D}_n \rightarrow \C_n(\F)$ be the homomorphism of algebras described in Section \ref{bcsect}, $a$ be the rotation by an angle $\frac{2\pi}{n}$, $b$ be the symmetry.

Let $g \circ f(\tau) = T \in M_n(\F)$. Then by the Cayley-Hamilton theorem $m(T)=\deg \mu_T(x) \leq n$. Since $g \circ f(\mu_{T}(\tau))=\mu_T(T)=0$, we get $\mu_{T}(\tau) \in \ker g \circ f$. Next, consider two cases separately.

First case. Let $n$ be odd. Then from Lemma \ref{kergf} it follows that $\ker g \circ f$ is one-dimensional. On the other hand, the kernel of a homomorphism of algebras is an ideal, which means $\mu_{T}(\tau)$ and $\mu_{T}(\tau)\tau$ are linearly dependent. Thus, $m(\F{\cal D}_{n})\leq n+1$.

Second case. Let $n$ be even. Then from Lemma \ref{kergf} it follows that $\ker g \circ f$ is two-dimensional. On the other hand, the kernel of a homomorphism of algebras is an ideal, which means $\mu_{T}(\tau)$, $\mu_{T}(\tau)\tau$, and $\mu_{T}(\tau)\tau^2$ are linearly dependent. Thus, $m(\F{\cal D}_{n})\leq n+2$.

\end{proof}

\begin{Remark}
The main conjecture regarding the lengths of group algebras in the case of dihedral groups is that $l(\F{\cal D}_{n})=n$ for all $n\geq 3$ over an arbitrary field.
Due to Theorem \ref{ldm}, to prove the conjecture, it is sufficient to obtain an estimate $m(\F{\cal D}_{n})\leq n+1$. However, using an estimate from Theorem \ref{mdn}, we get the same result as presented in Theorem \ref{lendn}. Nevertheless, estimating $m(\F{\cal D}_{n})$ allows us to demonstrate another application of the Theorem \ref{ldm} and the bicirculant representation, and the study of numerical characteristics of algebras is of interest in itself.  

\end{Remark}


\end{document}